\documentclass[11pt]{article}
\usepackage{amssymb, verbatim, amsmath}
\usepackage[all]{xy}
\usepackage{amsthm}
\usepackage[utf8]{inputenc}
\usepackage{dirtytalk}

\newcommand\blfootnote[1]{%
  \begingroup
  \renewcommand\thefootnote{}\footnote{#1}%
  \addtocounter{footnote}{-1}%
  \endgroup
}

\parskip \baselineskip

\newcommand{\bc}{\begin{center}}
\newcommand{\ec}{\end{center}}
\newcommand{\be}{\begin{enumerate}}
\newcommand{\ee}{\end{enumerate}}
\newcommand{\beq}{\begin{equation}}
\newcommand{\eeq}{\end{equation}}
\newcommand{\bi}{\begin{itemize}}
\newcommand{\ei}{\end{itemize}}
\newcommand{\bd}{\begin{description}}
\newcommand{\ed}{\end{description}}
\newcommand{\ba}{\begin{array}}
\newcommand{\bea}{\begin{eqnarray*}}
\newcommand{\eea}{\end{eqnarray*}}
\newcommand{\ea}{\end{array}}
\newcommand{\bt}{\begin{tabular}}
\newcommand{\et}{\end{tabular}}

\newcommand{\bmi}{\begin{minipage}}
\newcommand{\emi}{\end{minipage}}

\newcommand{\lb}{\linebreak}

\renewcommand{\thefootnote}{\fnsymbol{footnote}}
\newtheorem{thm}{Theorem}[section]
\newtheorem{defn}[thm]{Definition}
\newtheorem{pro}[thm]{Proposition}
\newtheorem{algo}[thm]{Algorithm}
\newtheorem{rem}[thm]{Remark}

\newtheorem{exa}[thm]{Example}
\newtheorem{cor}[thm]{Corollary}
\newtheorem{lem}[thm]{Lemma}

\newcommand{\Vector}[2]{\left(\begin{matrix} #1 \\ #2 \end{matrix} \right)}

\begin{document}

\bc {\bf\large The subspace structure of finite dimensional Beidleman near-vector spaces}\\[3mm]
{\sc  P Djagba \&  K-T Howell}

\it\small
Department of Mathematical Sciences,
Stellenbosch University,
Stellenbosch, 7600,\lb
South Africa\\
\rm e-mail:prudence@aims.ac.za, kthowell@sun.ac.za
\ec
 
\normalsize

\quotation{\small {\bf Abstract:}
 The subspace structure of Beidleman near-vector spaces is investigated. We characterise  finite dimensional Beidleman near-vector spaces and we classify the $R$-subgroups of finite dimensional Beidleman near-vector spaces. We provide an algorithm to compute the smallest $R$-subgroup containing a given set of vectors. Finally, we classify the subspaces of finite dimensional Beidleman near-vector spaces.  }

\normalsize

\section{Introduction}
The first notion of near-vector spaces was introduced by Beidleman in $1966$ \cite{beidleman1966near}. Subsequently, several researchers like Whaling, Andr\'e, and Karzel introduced a similar notion in different ways. Andr\'e near-vector spaces have been studied in many papers (for example \cite{andre1974lineare,howell2007contributions, howell2015subspaces}).  In this paper, we add to the theory  of near-vector spaces originally defined by  Beidleman. As in \cite{howell2015subspaces} for  Andr\'e near-vector spaces, we investigate the subspace structure of Beidleman near-vector spaces and highlight differences and similarities between these two types of near-vector spaces.

In section $2$ we give some preliminary material needed for this paper. In section $3$ we characterize finite dimensional Beidleman near-vector spaces. In section $4$ we investigate some properties of the substructures of Beidleman near-vector spaces. In section $ 5$ we classify the $R$-subgroups of the near-vector space $R^n$  containing a given set of vectors. We provide an algorithm and  Sage program for computations. Finally in section $6$ we classify all the subspaces of $R^n$.
\section{Preliminary material }

\subsection{Basic definitions and results}

\begin{defn}(\cite{meldrum1985near})
The triple $(R,+,\cdot)$ is a (left) nearring  if $(R,+)$ is a group,
 $(R,\cdot)$ is a semigroup, and  $a(b+c)= ab+ac$ for all $a,b,c \in R.$
\end{defn}

A nearfield is an algebraic structure similar to a skew-field, sometimes called a division ring, except that it has only one of the two distributive laws.  
\begin{defn} (\cite{pilz2011near})
  Let $R$ be nearring. If $ \big ( R^*=R \setminus \{0 \}, \cdot \big )$ is a group then $(R,+, \cdot)$ is called nearfield.
\label{th:t8}
 \end{defn}
 We will make use of left nearfields and right nearring modules in this paper.
 Dickson, Zassenhauss, Neumann, Karzel and Zemmer have shown by different methods that the additive group of a nearfield is abelian.
 \begin{thm}(\cite{pilz2011near})
 The additive group of nearfield is abelian.
 \label{tt}
 \end{thm}
To construct finite Dickson nearfields, we need two concepts: 
 \begin{defn} (\cite{pilz2011near})
A pair of numbers $(q,n) \in \mathbb{N}^2$ is called a Dickson pair if
$q$ is some power $p^l$ of a prime $p$, each prime divisor of $n$ divides $q-1$,  $q \equiv 3$ $ \text{mod } 4$ implies $4$ does not divide $n$.
\end{defn}

 \begin{defn}(\cite{pilz2011near})
Let $R$ be a nearfield and $\textit{Aut} (R,+,\cdot ) $ the set of all automorphisms of $N$. A map
\begin{align*}
\phi: \quad & R^* \to  \textit{Aut} (R,+,\cdot )  \\ 
& n \mapsto \phi_n
\end{align*}
is called a coupling map if for all $n,m \in R^*, \phi _n \circ \phi_m= \phi _{ \phi _n (m) \cdot n}.$
\end{defn}
The first  proper finite nearfield was discovered by  Dickson in $1905$. He distorted the multiplication of a finite field. For all pairs of Dickson numbers $(q,n)$, there exists some associated finite Dickson nearfields, of order $q^n$ which arise by taking the Galois field $GF(q^n)$ and changing the multiplication. Thus $DN(q,n)= (GF(q^n), +, \cdot) ^ { 
\phi} = \big ( GF(q^n), +, \circ \big )$. We will use $DN(q,n)$ to  denote a Dickson nearfield arising from the Dickson pair $(q,n)$. For more details on the construction of the new multiplication \say{$\circ$}  we refer the reader to \cite{dickson1905finite,pilz2011near}.
\begin{exa} (\cite{pilz2011near})
\label{ex2}Consider the field ($GF(3^{2})$, $+$, $\cdot$) with
\[GF(3^{2}) := \{0,1,2,x,1+x,2+x,2x,1+2x,2+2x\},\]
where $x$ is a zero of $x^{2}+1 \in \Bbb{Z}_{3}[x]$ with the new multiplication defined as

$$
a \circ b := \left\{\begin{array}{cc}
              a \cdot b     & \mbox{ if $a$  is a square in ($GF(3^{2})$, $+$, $\cdot$)}\\
              a \cdot b^3 & \mbox{ otherwise }
              \end{array}
\right.
$$
This gives the smallest finite  Dickson nearfield $DN(3,2):=(GF(3^{2})$, $+$, $\circ)$, which is not a field. Here is the table of the new operation $ \circ$ for $DN(3,2)$.

\[
 \begin{array}{r|ccccccccc}
\circ     & 0 & 1             & 2            & x     & 1+x  & 2+x  & 2x   & 1+2x & 2+2x \\ \hline
0         & 0 & 0             & 0            & 0          & 0         & 0         & 0         & 0         & 0\\ 
1         & 0 & 1             & 2            & x     & 1+x  & 2+x  & 2x   & 1+2x & 2+2x \\    
2         & 0 & 2             & 1            & 2x    & 2+2x & 1+2x & x    & 2+x  & 1+x \\    
x    & 0 & x        & 2x      & 2          & 1+2x & 1+x  & 1         & 2+2x & 2+x \\    
1+x  & 0 & 1+x      & 2+2x    & 2+x   & 2         & 2x   & 1+2x & x    & 1 \\    
2+x  & 0 & 2+x      & 1+2x    & 2+2x  & x    & 2         & 1+x  & 1         & 2x \\    
2x   & 0 & 2x       & x       & 1          & 2+x  & 2+2x & 2         & 1+x  & 1+2x \\    
1+2x & 0 & 1+2x     & 2+x     & 1+x   & 2x   & 1         & 2+2x & 2         & x \\    
2+2x & 0 & 2+2x     & 1+x     & 1+2x  & 1         & x    & 2+x  & 2x   & 2     
 \end{array}
\]

We will refer to this example  in later sections.
\end{exa}

 The concept of a ring module can be  extended to a more general concept called  a nearring module where the set of scalars is taken to be a nearring.
\begin{defn}
An additive group $(M,+)$ is called (right) nearring module over a (left) nearring $R$ if there exists a mapping,
\begin{align*}
\eta: \thickspace & M \times R \to M \\
& (m,r) \to mr
\end{align*} such that $m(r_1+r_2)=mr_1+mr_2$ and $m(r_1r_2)= (mr_1)r_2$ for all $r_1,r_2 \in R$ and $m \in M.$

We write $M_R$ to denote that $M$ is a  (right) nearring module over a (left) nearring  $R$.
\end{defn}
\begin{defn}(\cite{beidleman1966near})
A subset $A$ of a nearring module $M_R$ is called a $R$-subgroup if  $A$ is a subgroup of $(M,+),$ and  $AR= \lbrace ar \vert a \in A, r \in R \rbrace \subseteq A. $
\end{defn}
\begin{defn} (\cite{beidleman1966near})
A nearring module $M_R$ is said to be irreducible if $M_R$ contains no proper $R$-subgroups. In other words, the only $R$-subgroups of $M_R$ are $M_R$ and $\lbrace 0 \rbrace.$
\end{defn}
\begin{cor}(\cite{beidleman1966near})
Let $M_R$ be a unitary $R$-module. Then $M_R$ is irreducible if and only if $mR=M_R$ for every non-zero element $m \in M.$
\label{cor}
\end{cor}
\begin{defn} (\cite{beidleman1966near})Let $M_R$ be a nearring module.
$ N  $ is a submodule  of $M_R$ if :
\begin{itemize}
\item $ (N,+)$ is normal subgroup of $(M,+),$
\item $(m+n)r-mr \in N$ for all $m \in M, n \in N$ and $r \in R.$
\end{itemize}
\end{defn}

\begin{pro}(\cite{beidleman1966near})
Let $N$ be a submodule of $M_R.$ Then $N$ is a $R$-subgroup of $M_R.$
\label{pro}
\end{pro}
 Note that the converse of this proposition is not true in general. In his   thesis (\cite{beidleman1966near}, page $14$) Beidleman gives a counter example. However,
\begin{lem}
If $M_R$  is a ring module,  then the notions of $R$-subgroup and submodule of $M_R$ coincide.
\end{lem}
\begin{proof}
By Proposition \ref{pro}, every $R$-submodule is a $R$-subgroup. Let $H$ be a $R$-subgroup of $M_R.$ Then $hr \in H$ for all $h \in H$ and $r \in R.$ But $hr=(m+h)r-mr$ for all $m \in M.$ Hence $H$ is a submodule of $M_R.$
\end{proof}
\begin{thm}(\cite{beidleman1966near})
Let $R$ be a nearring that contains a right identity element $e \neq 0.$ $R$ is division nearring if and only if $R$ contains no proper $R$-subgroups.
\label{irre}
\end{thm}
\begin{rem}
Let $R$ be a nearfield. By Theorem \ref{irre}, $R_R$  is irreducible $R$-module. Thus $R$ contains only $\{ 0 \}$ and $R$ as submodules of $R_R.$
\end{rem}

\subsection{Internal direct sum of submodules}
\begin{defn}(\cite{beidleman1966near})
Let $\lbrace M_i \vert i \in I \rbrace $ be a collection of submodules  of the nearring module $M_R$.  $M_R$ is said to be a direct sum of the submodules $ M_i,$ for  $ i \in I, $ if  the additive group $(M,+)$ is a direct sum of the normal subgroups $ (M_i,+),$ for  $ i \in I $. In this case we write $M_R= \bigoplus _{i \in I} M_i.$
\end{defn}

\begin{pro}(\cite{beidleman1966near})
$M_R= \sum _ {i \in I} M_i$ and every element of $M_R$ has a unique representation as a finite sum of elements chosen from the submodules $M_i$ if and only if $ M_R= \sum _{i \in I} M_i$ and $M_k \cap \sum _{i \in I, i \ne k} M_i= \lbrace 0 \rbrace.$
\end{pro}

We also have that
\begin{pro}(\cite{beidleman1966near})
Let $\lbrace M_i \vert \thickspace  i \in I \rbrace $ be a collection of submodules  of the nearring module $M_R$. Then $M_R = \bigoplus _{i \in I} M_i$ implies that $M_R= \sum_{i \in I } M_i$ and the elements of any two distinct submodules permute.
\label{rp}
\end{pro}

%
%

According to the definition of a nearring module, we do not have distributivity of elements of $R$ over the elements of $M$. If we consider $M_R$ as direct sum of the  collection of submodules  $\lbrace M_i \vert \thickspace i \in I \rbrace$ of the nearring module $M_R$, then the following result will allow us to distribute the elements of $R$ over the elements which are contained in distinct submodules in the direct sum. The result is useful in the concept of Beidleman near-vector spaces.

\begin{lem}(\cite{beidleman1966near})Let $M_R = \bigoplus _{i \in I} M_i,$  $M_i$ is a submodule of $M_R.$ If $m =\sum_{i \in I} m_i$ where $m_i \in M_i$ and $r \in R$ then
\begin{align*}
mr= \big ( \sum_{i \in I} m_i \big ) r= \sum_{i \in I}( m_ir).
\end{align*}
\label{lemm}
\end{lem}

\subsection{Beidleman near-vector spaces}

In  \cite{andre1974lineare}, the concept of a vector space (or linear space) is generalised by Andr\'e to a less linear structure which he called a near-vector space. Andr\'e near-vector spaces use automorphisms in the construction, resulting in the right distributive law holding. However, for Beidleman near-vector spaces, we have the left distributive law holding and nearring modules are used in the construction.

\begin{defn}(\cite{beidleman1966near})
A nearring module $M_R$ is called strictly semi-simple if $M_R$ is a direct sum of irreducible submodules.
\end{defn}
We now have,
\begin{defn}(\cite{beidleman1966near}) Let $(M,+)$ be a group.
$M_R$ is called Beidleman near-vector space if $M_R$ is a strictly semi-simple $R$-module  where $R$ is a nearfield.
\end{defn}

 The simplest example of a Beidleman near-vector space is obtained when $M_R$ itself is an irreducible $R$-module.
\begin{lem} (\cite{beidleman1966near}) 
Let $R$ be a nearfield and $M_R$ an irreducible $R$-module. Then $M_R$ is a Beidleman near-vector space. Moreover, 
\begin{align*} 
M_R \cong R_R.
\end{align*}
\end{lem}

As with  vector spaces, we have the notion of a basis and dimension.

\begin{defn}(\cite{beidleman1966near})
A non-empty subset $X$ of a near-vector space $M_R$ is called a basis if $X$ is a spanning set for $M_R$ and  the representation of the elements of $M_R$ as a linear combinations of the elements of $X$ is unique.
\end{defn}
\begin{thm}(\cite{beidleman1966near})
If $M_R$ is a near-vector space, then $M_R$ has a basis.
\end{thm}
Finally we have,
\begin{defn}(\cite{beidleman1966near})
If $M_R$ is a near-vector space over $R,$ then the cardinality of any basis is called the dimension of $M_R$ and is denoted by $\dim M_R$.
\end{defn}

%
%
%
%
%
%

\section{Finite dimensional  Beidleman near-vector spaces}
In \cite{van1992matrix} van der Walt characterized finite dimensional Andr\'e near-vector spaces. In this section we do the same for  finite dimensional Beidleman near-vector spaces. We will see that finite dimensional Beidleman near-vector spaces are closest (in terms of structure) to traditional finite dimensional vector spaces. Let $r_1, \ldots,r_n \in R.$ We use $(r_1, \ldots,r_n )^T$ to denote  the transpose of the row vector $(r_1, \ldots,r_n ).$
\begin{thm}

Let $R$ be a (left) nearfield and $M_R$  be a right nearring module. $M_R$ is a finite dimensional near-vector space if and only if $M_R \cong R^n$ for some   positive integer $n = \dim M_R.$
\label{thm1}
\end{thm}
\begin{proof}
Let us consider the $R$-modules $M_R$ and $R^n_R$ with the action given as 
$ M \times R \to M$ such that $(m,r) \mapsto mr$. Since $M_R$ is a Beidleman near-vector space,  $M_R = \bigoplus _{i =1 }^n M_i$ where $M_i$ for $ i \in \lbrace 1,\ldots,n\rbrace$ are non-zero irreducible $R$-submodules of $M_R.$ Then by  Corollary \ref{cor}, there exists $0 \neq m_i \in M_i$ such that $m_iR=M_i$ for $ i \in \lbrace 1,\ldots,n\rbrace$.  Hence $M_R= \bigoplus_{i=1}^n m_iR.$ Then it is not difficult to see that $ B=\lbrace m_1, \ldots, m_n \rbrace $ is a basis of $M_R.$
Let us consider the map
\begin{align*}
\phi :& \thickspace R^n \to M \\
& (r_1, \ldots,r_n)^T \mapsto  \sum_{i=1}^{n} m_ir_i.
\end{align*}
Let $(r_1, \ldots,r_n)^T, (r'_1, \ldots,r'_n)^T \in R^n.$
We have
\begin{align*}
\big ((r_1, \ldots,r_n)^T + (r'_1, \ldots,r'_n)^T  \big )   \phi & = \sum_{i=1}^{n} m_i(r_i+ r'_i) 
=\sum_{i=1}^{n} (m_ir_i+ m_ir'_i) \\
&=\sum_{i=1}^{n} m_ir_i+ \sum _{i=1}^{n} m_ir'_i  \\
& = \big ((r_1, \ldots,r_n)^T \big ) \phi + \big ( (r'_1, \ldots,r'_n)^T  \big ) \phi.
\end{align*}
Let $ r \in R$  and  $ (r_1, \ldots,r_n)^T \in R^n.$ Using   Lemma \ref{lemm}, we obtain

\begin{align*}
\big (  (r_1, \ldots,r_n)^T r \big )   \phi & = \sum_{i=1}^{n} (m_ir_i)r  =  \big ( \sum_{i=1}^{n} m_ir_i \big ) r = \big (  (r_1, \ldots,r_n)^T  \big )   \phi r.
\end{align*}
Since $B$ is a basis of $M_R$ then $ \big ( (r_1, \ldots,r_n)^T \big )   \phi =0 \Rightarrow \sum_{i=1}^{n} m_ir_i =0 \Rightarrow r_1=r_2= \ldots=r_n=0$.  We deduce that,
\begin{align*}
Ker \phi  =  \big \lbrace (r_1, \ldots,r_n)^T  \in R^n  \vert ((r_1, \ldots,r_n)^T) \phi = 0   \rbrace 
=  \big \lbrace (0, \ldots,0)^T \big \rbrace.
\end{align*}
It follows that $\phi$ is injective.

 Let $m \in M_R.$ Since $B$ is a basis of $M_R$, there exists $r_1, \ldots, r_n \in R$ such that $m = \sum_{i=1}^{n} m_ir_i=\big ( (r_1, \ldots,r_n)^T \big )   \phi  $. It follows that $\phi $ is surjective. Hence $\phi$ is bijective map. 

\end{proof}

\begin{rem}Let $R$ be a nearfield.
\begin{itemize}
\item By van der Walt's theorem \cite{van1992matrix}, $(R^n,R) $ is an Andr\'e near-vector space with the scalar multiplication defined by
\begin{align*}
\alpha  (x_1, \ldots,x_n) = \big ( \psi_1( \alpha) x_1 , \ldots, \psi_n( \alpha) x_n \big ),
\end{align*} for all $ \alpha \in R$ and $ (x_1, \ldots,x_n) \in R^n$
where the  $\psi_i$ for $i  \in \lbrace 1,\ldots,n\rbrace$ are multiplicative  automorphisms of $R^*=R \setminus \{0 \}$. Note that the action of the scalars on the vectors are on the left, whereas for Beidleman the action of the scalars on the vectors are on right. Also by van der Walt's construction theorem \cite{van1992matrix}  we can take different nearfields to construct finite dimensional Andr\'e near-vector spaces, as long as the nearfields are multiplicatively isomorphic. To construct Beidleman near-vector spaces, $n$ copies of the same nearfield are used in the construction and we can use the automorphism of nearfields to define the scalar multiplication. In later, we will only focus on the identity automorphism of nearfield as the scalar multiplication.
\item 
If $R$ is a field then $(R^n,R)$  is both an Andr\'e and Beidleman near-vector space, and both coincide to a vector space. Here we are taking all the $\psi_i=id$ for $i \in \lbrace 1,\ldots,n\rbrace$. 
\end{itemize}
\end{rem}
\begin{rem} Let $I$ be an index set and $M_R$ a Beidleman near-vector space. Then $M_R= \bigoplus_{i \in I}M_i,$ where the $M_i, i \in I,$ are irreducible submodules of $M_R.$ By Corollary \ref{cor} there exists $0 \neq m \in M_i$ such that $ M_R= \bigoplus_{i \in I}m_iR.$ Hence $ \{ m_i | \thickspace i \in I \}$ is a basis for $M_R$. If $I$ is infinite then $M_R$ has infinite dimension and we can use the same procedure in the proof of Theorem \ref{thm1} to  show that $M \cong R^{\dim M_R}.$
\end{rem}

\section{The subspace structure of Beidleman near-vector spaces}

	In this section we investigate some properties of the subspace structure of  Beidleman near-vector spaces. We find that with regard to  subspace structure, Andr\'e near-vector spaces are closest to traditional vector spaces.

As with  Andr\'e near-vector spaces, we deduce the following.
\begin{lem}
Let $M_R$ be a  Beidleman near-vector space. Then $(M,+)$ is abelian.
\label{r}
\end{lem}
\begin{proof}
If $M_R$ is of finite dimension then by Theorem \ref{thm1},  $M_R \cong R^n$ where $\dim M_R=n.$ Since $R$ is a nearfield, by Theorem \ref{tt} $(R,+)$ is abelian. It follows that $(R^n,+)$ is abelian. Hence $(M,+)$ is abelian. If $M$ is of infinite dimension, $M \cong R^{dim M}.$ So $(M,+)$ is abelian.
\end{proof}

Let $M_R$ be a Beidleman near-vector space. $N$ is a subspace of $M_R$ if $\emptyset \neq N \subset M$ and $N$ is also a Beidleman near-vector space.

From Lemma \ref{r}, we  can give an equivalent definition of a subspace as follows.
\begin{defn}Let $M_R$ be a Beidleman near-vector space. 
$\emptyset \neq N \subset M$ is a subspace of $M_R$ if 
\begin{itemize}
\item[(1)] $(N,+)$ is a subgroup of $(M_R,+)$,
\item[(2)] $(m+n)r-mr \in N$ for all $m \in M, n \in N$ and $r \in R.$
\end{itemize}
\end{defn}
\begin{rem}
According to \cite{beidleman1966near},  $(N,+)$ should be a normal subgroup of $(M,+)$. But since $(M,+)$ is abelian, we do not need the normality again. 
\end{rem}

\begin{lem}Let $M_R$ be a Beidleman near-vector space.
Let $M_1$ and $M_2$ be subspaces of $M$. Then $ M_1 \cap M_2$ is also subspace of $M$.
\end{lem}
\begin{proof}
It is not difficult to see that $\big (M_1 \cap M_2,+  \big ) $ is a subgroup of $(M,+)$. Let $ m \in M, n \in  M_1 \cap M_2$ and $r \in R.$ Then  $(m+n)r-mr \in M_1$ and $(m+n)r-mr \in M_2$. Thus $(m+n)r-mr \in M_1 \cap M_2.$
\end{proof}
By Lemma \ref{r}, we now deduce the following.
\begin{lem}Let $M_R$ be a Beidleman near-vector space.
Let $M_1$ and $M_2$ be subspaces of $M$. Then $ M_1+ M_2$ is also a subspace of $M$ where $M_1+M_2= \lbrace m_1+m_2 \vert m_1 \in M_1, m_2 \in M_2 \rbrace.$
\end{lem}
\begin{proof}~
It is straightforward to check that $ \big ( M_1+M_2  ,+\big )$ is a subgroup of $(M,+)$. Furthermore, let $m \in M ,l \in M_1+M_2$ and $r \in R.$ Then $l=m_1+m_2$ for some $m_1 \in M_1$ and $m_2 \in M_2.$ We have
\begin{align*}
(m+l)r-mr & = (m+m_1+m_2)r-mr  \\
&= \big ( (m+m_1)+ m_2 \big ) r -mr \\
&=  \big ( (m+m_1)+ m_2 \big )  r -  (m+m_1)r + (m+m_1)r -mr.
\end{align*}
Since $\big ( (m+m_1)+ m_2 \big )  r -  (m+m_1)r \in M_2, \thickspace$  $(m+m_1)r -mr \in M_1$ and $(M,+)$ is abelian it follows that $(m+l)r-mr \in M_1+M_2.$ 
\end{proof}

For vector spaces and  Andr\'e near-vector spaces, a non-empty subset is subspace if and only if it is closed under addition and scalar multiplication \cite{howell2015subspaces}. For a Beidleman near-vector spaces we only have
\begin{lem}
If $ \emptyset \neq  N_R$ is a subspace of $M_R$ then $N_R$ is closed under addition and scalar multiplication.
\end{lem}
\begin{proof}
If $ N_R$ is a subspace of $M_R$, then $\big ( N_R,+ \big )$ is a subgroup of $\big ( M_R,+ \big )$ and for all $m \in M, r \in R, n \in N$ we have $(m+n)r-mr \in N.$ In particular for $m=0$, we obtain that $nr \in R.$ 
\end{proof}
Note that every subspace is a $R$-subgroup but  being closed under addition and scalar multiplication does not in general give a subspace.  We now provide a counter example.
\begin{exa}
Let $R=DN(3,2)$ be the finite Dickson nearfield that arises from the Dickson pair $(3,2)$. Then $(R^2,R)$ is a Beidleman near-vector space.
Let us consider
\begin{align*}
T=  \big \{ \Vector{1}{x}r | r \in R  \big \}= \big <  \Vector{1}{x} \big >. 
\end{align*}
Note that $T$ is a $R$-subgroup of $R^2$ but not a subspace of $R^2$. Indeed by Example \ref{ex2}, 
Let $\Vector{1}{x+1} \in R^2, \Vector{1}{x} \in T $ and $x \in R.$ We have
\begin{align*}
\left(  \left( {\begin{array}{cc}
   1  \\
  x+1
  \end{array} } \right)  + \left( {\begin{array}{cc}
   1  \\
  x
  \end{array} } \right)  \right) \circ  x  - \left( {\begin{array}{cc}
  1 \\
   x+1 
  \end{array} } \right) \circ x &= \left( {\begin{array}{cc}
   2x  \\
  (2x+1) \circ x -(x+1) \circ x
  \end{array} } \right) \\
  &=\left( {\begin{array}{cc}
   2x \\
  -1
  \end{array} } \right)  \notin T.
\end{align*}
Hence $T$ is not a subspace of $R^2$.

\label{example}
\end{exa}

%
%
%
%

\section{Classification of the $R$-subgroups of $R^n$}

\blfootnote{The authors wish to express their appreciations to Georg Anegg for his jointly work.}

Let $(R,+,\circ,0,1)$ be a Dickson nearfield for the Dickson pair $(q,m)$ with $m>1$. We know that the distributive elements of $R$ form a subfield of size $q$. Also, not all elements are distributive, thus there are elements $ \lambda \in R$ such that $(\alpha+\beta)\circ \lambda \neq \alpha \circ \lambda + \beta \circ \lambda$ for all $\alpha, \beta \in R$. We will use  $R_d$ to denote the set of all distributive elements of $R$, i.e.,
\begin{equation*}
R_d = \lbrace z \in R | \thickspace (x+y)\circ z=x \circ z+y \circ z \thickspace \mbox{for all x,y} \in R \rbrace,
\end{equation*} 
where from now on, we shall simply use concatenation instead of $\circ$.

Consider the $R$-module $R^n$ (for some fixed $n\in \mathbb{N}$) with componentwise addition and  scalar multiplication $R^n \times R \to R^n$ given by $(x_1,x_2,\ldots,x_n) r= (x_1r,x_2r,\ldots,x_nr)$ for $(x_1,x_2,\ldots,x_n) \in R^n$ and $r \in R$. A $R$-subgroup $S$ of $R^n$ is a subgroup of $(R^n,+)$ that is closed under scalar multiplication.
Let $v_1,v_2,\ldots,v_k \in R^n$ be a finite number of vectors. The smallest subspace of $R^n$ containing $ \{ v_1,v_2,\ldots,v_k\}$ is called the \textit{span} of $v_1,v_2,\ldots,v_k$ and is denoted by $span(v_1,\ldots,v_k)$.

In analogy to span, we introduce the notion of \textit{gen}.
\begin{defn}
Let $v_1,v_2,\ldots,v_k \in R^n$ (for some $k\in \mathbb{N}$) be a finite number of vectors. We define $gen(v_1,\ldots,v_k)$ to be the smallest $R$-subgroup of $R^n$ containing $\{ v_1,v_2,\ldots,v_k \}$.
\end{defn}

Our first aim is to find an explicit description of $gen(v_1,\ldots,v_k)$.

Let $LC_0(v_1,v_2,\ldots,v_k):=\{ v_1,v_2,...,v_k\}$ and for $n\geq0$, let $LC_{n+1}$ be the set of all linear combinations of elements in $LC_n(v_1,v_2,\ldots,v_k)$, i.e.
\begin{equation*}
	LC_{n+1}(v_1,v_2,\ldots,v_k)=\left \{ \sum_{w \in LC_n} w \lambda_w |\lambda_w \in R  \right\}.
\end{equation*}
We will  denote $LC_n(v_1,v_2,\ldots,v_k)$ by $LC_n$ for short when there is no ambiguity with regard to the initial set of vectors. 
\begin{thm}Let $v_1,v_2,\ldots,v_k \in R^n$.
	We have 
	\begin{equation*}
	gen(v_1,\ldots,v_k)=\bigcup_{i=0}^\infty LC_i.
	\end{equation*}
	\label{th1}
\end{thm}

\begin{proof}
We need to show that $ \bigcup_{i=0}^\infty LC_i$ is a $R$-subgroup of $R^n$ and for any $R$-subgroup $S$ containing $v_1,v_2,...,v_k,$ we have $ \bigcup_{i=0}^\infty LC_i \subseteq S.$

The zero vector  $0 _{R^n}$ of $(R^n,+)$ is an element of $\bigcup_{i=0}^\infty LC_i.$ Let $x,y \in \bigcup_{i=0}^\infty LC_i.$ We distinguish two cases: 
\begin{itemize}
\item[Case 1:] $x,y \in LC_n$ for some $n \in \mathbb{N}$. So $x-y \in LC_{n+1}.$
\item[Case 2:] $x \in LC_k $ and $y \in LC_n$ for some $k < n.$ Since $k<n,$ we have $LC_k \subseteq LC_n$ by definition. Hence $x-y \in LC_n$.
So $x-y \in \bigcup_{i=0}^\infty LC_i.$
\end{itemize}
Therefore $ \big (\bigcup_{i=0}^\infty LC_i,+ \big )$ is a subgroup of $\big ( R^n,+ \big )$. 

Let $x \in \bigcup_{i=0}^\infty LC_i$ and $r \in R$. Then $x \in LC_n$ for some $ n \in \mathbb{N}$. For $n=0$, $x=v_i$ for some $i \in \{ 1, \ldots,k\}$, so $v_ir \in LC_1$ . For $ n \geq 1,$ we have $ x =\sum_{w \in LC_{n-1}} w \lambda_w,$ and thus $xr= \big ( \sum_{w \in LC_{n-1}} w \lambda_w \big )r= \sum_{w \in LC_{n-1}} w \lambda_w r  \in LC_{n}. $

It remains to show that for any $R$-subgroup $S$ containing $v_1,\ldots,v_k$ we have $\bigcup_{i=0}^\infty LC_i \subseteq S. $ It is sufficient to show that for all $ n \in \mathbb{N},$ $LC_n \subseteq S.$ We use induction on $n$. For $n=0$ we have $LC_0 \subseteq S.$ Assume that $LC_k \subseteq S$ for $ k \in \mathbb{N}$. Let $ x \in LC_{k+1}.$ Then  $ x =\sum_{w \in LC_k} w \lambda_w,$ where $\lambda_w \in R.$ But $w  \in  LC_k \subseteq S.$ So $w \lambda_w  \in S$ since $S$ is a $R$-subgroup (closed under addition and scalar multiplication). Therefore $\sum_{w \in LC_n} w \lambda_w \in S$. Hence $x \in S.$
\end{proof}

In the following propositions we give some basic properties of  \textit{gen}.

\begin{pro} Let $k \in \mathbb{N}$ and $T$ be a finite set of vectors in $R^n.$ We have,
\begin{align*}
LC_n \big ( LC_k (T) \big )= LC_{n+k}(T) \thickspace \mbox{for all} \thickspace n \in \mathbb{N}.
\end{align*}
\end{pro}
The proof is not difficult and uses induction on the positive integer $n$.

\begin{pro} Let $S$ and $T$ be finite sets of vectors of $R^n.$ The following hold:
\begin{enumerate}
\item[(1)] $S \subseteq gen (S) \subseteq span(S),$
\item[(2)] If $S \subseteq T$ then $gen(S) \subseteq gen(T),$
\item[(3)] $gen(S \cap T) \subseteq gen(S) \cap gen(T),$
\item[(4)] $gen(S) \cup gen(T) \subseteq gen (S \cup T),$
\item[(5)] $gen \big (gen (T) \big )= gen(T).$
\end{enumerate}
\end{pro}
\begin{proof} $(1),(2),(3)$ and $(4)$ are straightforward. Indeed $gen(T) \subseteq gen  \big ( gen(T) \big )$. Also for all $n \in \mathbb{N},$ we have  that $LC_n \big (  gen(T)\big) \subseteq gen(T).$ Hence $gen \big (gen (T) \big ) \subseteq gen(T).$
\end{proof}
We want to give a  description of $gen(v_1,\ldots,v_n)$ in terms of the basis elements. In the following lemmas, we first derive  analogous results of row-reduction in vector spaces. The first lemma  follows directly from Theorem \ref{th1} and we state it without proof.
\begin{lem}
	For any permutation $\sigma$ of the indices $1,2,...,k$, we have
	\begin{equation*}
	gen(v_1,\ldots,v_k)=gen(v_{\sigma(1)} ,\ldots,v_{\sigma(k)}).
	\end{equation*}
	\label{l4}
\end{lem}

%
We can also show that
\begin{lem}
	If $ 0 \neq \lambda \in R $, then
	\begin{equation*}
	gen(v_1,\ldots,v_k)=gen(v_1  \lambda,\ldots,v_k).
	\end{equation*}
	\label{l3}
\end{lem}
\begin{proof}
Let	$gen(v_1,\ldots,v_k)=\bigcup_{i=0}^\infty LC_i $ and $ gen(v_1  \lambda,\ldots,v_k)=\bigcup_{i=0}^\infty LC'_i.$
Let $x \in \bigcup_{i=0}^\infty LC_i.$ Then $x \in LC_n$ for some $n \in \mathbb{N}$. Clearly $LC_0 \subseteq LC'_1$	and $ LC_1 = \lbrace v_1  \alpha_1 + v_2 \alpha_2 +\cdots+ v_k \alpha_k| \alpha_1, \ldots, \alpha_k \in R \rbrace$. Say $x \in LC_1.$ Since $ \lambda \neq 0$ there exists $\lambda' \in R$ such that $ \lambda \lambda'=1$. So $x=v_1  \alpha_1 + v_2 \alpha_2 +\cdots+ v_k \alpha_k =v_1 (\lambda \lambda')\alpha_1 + v_2 \alpha_2+\ldots+ v_k \alpha_k =(v_1 \lambda) \lambda'\alpha_1  + v_2 \alpha_2+\ldots+ v_k \alpha_k.$ Then $x \in LC'_1.$ Thus we see if $x \in LC_n$ for some $n \neq 0$ in the expression of $x$ we have
\begin{align*}
v_1 \alpha_1= v_1 1 \alpha_1 = v_1 (\lambda \lambda')\alpha_1 =(v_1 \lambda) (\lambda'  \alpha_1).
\end{align*}
Thus $x \in LC'_n$. So $ x \in \bigcup_{i=0}^\infty LC'_i $. Thus $x \in gen(v_1  \lambda,\ldots,v_k).$
	 
 Let $x \in \bigcup_{i=0}^\infty LC'_i .$ Then $x \in LC'_n$ for some $n \in \mathbb{N}$. In fact $LC'_0= \lbrace  v_1 \lambda, v_2, \ldots,v_k \rbrace$ and we have that $LC'_0 \subseteq LC_1$. Also $ LC'_1 = \lbrace (v_1 \lambda) \alpha_1 + v_2 \alpha_2 +\cdots+ v_k \alpha_k | \alpha_1, \ldots, \alpha_k \in R \rbrace$. Let $x \in LC'_1$ then $x=(v_1 \lambda) \alpha_1 + v_2 \alpha +\ldots+ v_k \alpha_k =v_1 (\lambda \alpha_1) + v_2 \alpha_2+\ldots+ v_k \alpha_k .$ It follows that $x \in LC_1.$ Thus we see that if $x \in LC'_n$ for some $n \neq 0$ in the  expression of $x$ we have,
\begin{align*}
(v_1 \lambda ) \alpha_1 = v_1 (\lambda \alpha_1) \thickspace \mbox{by  the associativity of the multiplication}.
\end{align*} It follows that  $x \in LC_n$. Thus $ x \in \bigcup_{i=0}^\infty LC_i$. Hence $x \in gen(v_1,\ldots,v_k). $ 
\end{proof}
\begin{lem}
	For any scalars $\lambda_2, \lambda_3,\ldots , \lambda_k \in R$, we have 
	\begin{equation*}
	gen(v_1,\ldots,v_k)=gen(v_1 + \sum_{i=2}^{k} v_i \lambda_i,v_2, \ldots,v_k).
	\end{equation*}
	\label{l2}
\end{lem}
\begin{proof}
By Theorem \ref{th1} we can write 	
\begin{align*}
	 gen(v_1,\ldots,v_k)=\bigcup_{i=0}^\infty LC_i \thickspace \mbox{and} \thickspace  gen(v_1 + \sum_{i=2}^{k} v_i \lambda_i,v_2, \ldots,v_k)=\bigcup_{i=0}^\infty LC'_i.
	\end{align*}
We have $LC_0 = \lbrace  v_1,\ldots,v_k \rbrace \thickspace \mbox{and} \thickspace 
	LC_1 = \lbrace \sum_{i=1}^{k} v_i \alpha_i | \alpha_i \in R \rbrace .$ Also $
		LC'_0 = \lbrace  v_1 + \sum_{i=2}^{k} v_i \lambda_i,v_2,\ldots,v_k \rbrace \thickspace \mbox{and} \thickspace
	 LC'_1 = \lbrace (v_1+ \sum_{i=2}^{k} v_i  \lambda_i) \beta_1+ \sum_{i=2}^{k} v_i \beta_i | \alpha_i , \beta_i \in R \rbrace. $
We proceed by induction. Clearly $ v_2, \ldots,v_k \in LC'_0 \subseteq LC'_1.$ Since
	\begin{align*}
	v_1= \big  (v_1+ \sum_{i=2}^{k} v_i \lambda_i \big )- \sum_{i=2}^{k} v_i \lambda_i,
	\end{align*} we have $v_1 \in LC'_1.$ So $LC_0 \subseteq LC'_1$. Assume that $LC_m \subseteq LC'_{m+1}$ for some $m \in \mathbb{N}.$  We need to show that $LC_{m+1} \subseteq LC'_{m+2}.$ Let $x \in LC_{m+1}. $ Then
	\begin{align*}
	x = \sum_{w \in LC_m} w \lambda_w =\sum_{w \in LC'_{m+1}} w \lambda_w.
\end{align*}	 It follows that $x \in LC'_{m+2}.$
For the other inclusion,  we also argue by induction. Clearly $v_2, \ldots,v_k \in LC_1$ and $v_1+ \sum_{i=2}^{k} v_i \lambda_i \in LC_1,$ so $LC'_0 \subseteq LC_1.$ Let $x \in LC'_1,$ Then $x=  \big (v_1+ \sum_{i=2}^{k} v_i \lambda_i  \big ) \beta_1 +  \sum_{i=2}^{k} v_i \beta_i.$ Then $x \in LC_2.$ Assume that   $LC'_m \subseteq LC_{m+1}$ for some $m \in \mathbb{N}.$ We need to show that $LC'_{m+1} \subseteq LC_{m+2}.$ Let $x \in LC'_{m+1}. $ So $x = \sum_{w \in LC'_m} w \lambda_w =\sum_{w \in LC_{m+1}} w  \lambda_w.$ Hence $x \in LC_{m+2}.$
\end{proof}
We need one more lemma which follows directly from Theorem \ref{th1} and we state it without proof.
\begin{lem}
If $w \in gen(v_1,\ldots,v_k)$, then
	\begin{equation*}
	gen(v_1,\ldots,v_k)=gen(w, v_1,v_2, \ldots,v_k).
	\end{equation*}
	\label{l1}
\end{lem}

%
%
%

Given a set of row vectors $v_1,\ldots,v_k$ in $R^n,$ arranged in a matrix $U $ of size $k \times n$,  $gen(v_1,\ldots,v_k)$ constructed from the matrix $U$ does not change under elementary row operations (swopping rows, scaling rows, adding multiples of a row to another). Recall that two matrices are said to be row equivalent if one can be transformed to the other by a sequence of elementary row operations.
\begin{lem}
Suppose that $k \times n$ matrices $V=  (v_1, \ldots, v_k)^T$ and $W= (w_1,\ldots,w_k)^T$ are row equivalent (with the rows $v_1,\ldots,v_k, w_1,\ldots,w_k \in R^n$). Then 
\begin{equation*}
gen(v_1,\ldots,v_k)=gen(w_1,\ldots,w_k)
\end{equation*}
\end{lem}
\begin{proof}
This proof follows from Lemmas \ref{l4}, \ref{l3}, \ref{l2} and \ref{l1}.
\end{proof}
We will use  $ \{ v_1,\ldots,\hat{v_i},\ldots,v_k \} $ to denote the fact that the vector $v_i$ has been  removed from the set of vectors $\{ v_1,\ldots, v_k \}.$
\begin{defn} Let $M_R$ be a Beidleman near-vector space.
Let $V=\{v_1,\ldots, v_k\}$ be a finite set of vectors of $M_R$. We say $V$ is $R$-linearly dependent if there exists $v_i\in V$ such that $v_i\in gen(v_1,\ldots,\hat{v_i},\ldots,v_k)$. We say that $V$ is $R$-linearly independent if it is not $R$-linearly dependent.
\end{defn}
A matrix $M$ consisting of $k$ rows and $n$ columns will be denoted by $M=(m_i^j)_{1\leq i \leq k , 1\leq j \leq n }$ where $m_i^j$ is the entry in the $i$-th row and $j$-th column.

\begin{defn}
The $R$-row space of a matrix is the set of all possible $R$-linear combinations of its row vectors.
\end{defn}
Thus the $R$-row space of a given matrix $M$ is the same as the \textit{gen} of the rows of $M$. We now turn to the classification of the smallest $R$-subgroup containing a given set of vectors in $R^n$, the main result of this section. 
\begin{thm}
Let  $v_1,\ldots,v_k$ be vectors in $R^n$. Then
\begin{align*}
gen(v_1,\ldots,v_k)= \bigoplus_{i=1}^{k'}u_iR,
\end{align*}where the $u_i$ (obtained from $v_i$ by an explicit
procedure) for $i \in \lbrace 1, \ldots,k' \rbrace $ are the rows of  the matrix  $U=\big(u_{i}^j  \big) \in R^{k' \times n}$ whose columns have exactly one non-zero entry.
\label{th2}
\end{thm}
\begin{proof}
Given a particular set of vectors $v_1,\ldots,v_k$, arrange them in a matrix $V$ whose $i$-th row is composed of the components of $v_i$, i.e., $V=(v_i^j)$ where $1 \leq j \leq n$. Then $gen(v_1,\ldots,v_k)$ is the $R$-row space of $V$, which is a $R$-subgroup of $R^n$. We can then do the usual Gaussian  elimination on the rows. According to the previous lemmas, the $gen$ spanned by the rows will remain unchanged with each operation (swopping rows, scaling rows, adding multiples of a row to another). When the algorithm terminates, we obtain a matrix $W \in R^{k \times n }$ in reduced row-echelon form (denoted by $RREF(V)$). Let the non-zero rows of $W$ be denoted by $w_1,w_2,\ldots,w_t$ where $t \leq k$.
\begin{itemize}
\item[Case 1:] Suppose that every column has at most one non-zero entry, then
\begin{equation*}
	gen(v_1,\ldots,v_k)=gen(w_1,\ldots,w_t)=w_1R+w_2R+\cdots+w_tR
\end{equation*} 
where the sum is direct. In this case we are done. 
\item[Case 2:] Suppose that the $j$-th column is the first column that has two non-zero entries, say $w_r^j\neq 0 \neq w_s^j$ with $r<s$, (we necessarily have $r,s\leq j$) where $w_r^j$ is the $j$-th entry  of row $w_r$ and $w_s^j$  the $j$-th entry  of row $w_s.$
Let $ \alpha, \beta, \gamma \in R$ such that $(\alpha +\beta) \lambda \neq \alpha  \lambda  + \beta  \lambda .$
We apply what we will call the \say{\textsl{distributivity trick}}:

Let $\alpha'= (w_r^j)^{-1}\alpha$ and $\beta'= (w_s^j)^{-1}\beta$. Then consider the new row
\begin{equation*}
	\theta=(w_r \alpha' + w_s  \beta') \lambda - w_r  (\alpha' \lambda)-w_s (\beta ' \lambda). 
\end{equation*}
Since $\theta \in LC_2(w_r,w_s)$ we have $\theta\in gen(w_1,\ldots,w_t)$.

 For $ 1 \leq l < j,$ either $w_r^l$ or $w_s^l$ is zero because the $j$-th column is the first column that has two non-zero entries, thus $\theta^l=0$. Note that by the choice of $\alpha,\beta,\lambda$, we have 
\begin{align*}
\theta^j &=(w_r^j)\alpha ' + w_s^j \beta') \lambda - (w_r^j\alpha ') \lambda- (w_s^j\beta ') \lambda \\
&=\big ( w_r^j (w_r^j)^{-1} \alpha + w_s^j(w_s^j)^{-1}  \beta \big ) \lambda- \big (  w_r^j ( w_r^j )^{-1} \alpha \big ) \lambda - \big (  w_s^j ( w_s^j )^{-1} \beta \big ) \lambda \\
&=(\alpha +\beta) \lambda - \alpha \lambda - \beta \lambda \neq 0.
\end{align*}
It follows  that $\theta^j\neq 0$.
Hence $ \theta =(0, \ldots, 0,\theta^j,\theta^{j+1},  \ldots, \theta^n )$. We now multiply the row $\theta$ by $ (\theta^j)^{-1},$ obtaining the row $\phi=(0, \ldots, 0,1,\theta^{j+1}(\theta^j)^{-1},  \ldots, \theta^n (\theta^j)^{-1}) \in gen(w_1,\ldots,w_k)$ where  $\phi ^j=1$ is  the pivot that we have created.

 As a next step, we form a new matrix of size $(t+1) \times n $ by adding $\phi$ to the rows $w_1, \ldots,w_t.$ On  this augmented matrix we replace the rows $w_r, w_s$  with $ y_r=w_r-(w_r^j) \phi, y_s=w_r-(w_s^j) \phi$ respectively. This yields another new matrix  composed of  the rows $w_1,\ldots,w_{r-1},y_r, \ldots, y_s, \phi, w_{s+1},  \ldots,w_t $ which has only one non-zero entry in the $j$-th column. By Lemma \ref{l1}, the  \textit{gen} of the rows of the augmented matrix is the  \textit{gen} of the rows of $W$ (which in turn is  $gen(v_1,\ldots,v_k)$).
Hence
\begin{align*}
gen(v_1, \ldots,v_k) & = gen(w_1,\ldots,w_r, \ldots, w_s, \ldots,w_t)\\
&= gen(w_1,\ldots,y_r, \ldots, y_s, \phi, \ldots,w_t).
\end{align*}

Continuing this process, we can eliminate all columns with more than one non-zero entry. Let the final matrix have rows $u_1,u_2,\ldots, u_{k'}$. Then
\begin{equation*}
	gen(v_1,\ldots,v_k)  =gen(w_1,\ldots,w_t)=gen(u_1,\ldots,u_{k'})=u_1R+u_2R+\ldots+u_{k'}R, 
\end{equation*} 
where the sum is direct.
\end{itemize}
\end{proof}

\begin{rem} ~
\begin{itemize} 
\item The procedure described in the proof of Theorem \ref{th2} will be called expanded Gaussian elimination (eGe algorithm).
\item Let  $v_1, \ldots, v_k \in R^n$ be vectors arranged in a matrix $ V \in R^{k \times n}.$ Suppose $RREF(V) \\ =W$ and let $N(w^j)$ be the number of non-zero entries of the $j$-th column of $W$ where $1 \leq j \leq n $.  Let $t$ be the number of non-zero rows of $W.$
If $v_1, \ldots, v_k$ are $R$-linearly independent and $ \sum_{j=1}^n N(w^j) \geq n +k-t$ then  by Theorem \ref{th2}, we shall add $k-t+1$ new rows to the $t$ non-zero rows. Hence we  necessarily have that  $k \leq k' \leq n.$
\end{itemize}

\end{rem}
\begin{exa}
Let $R=DN(3,2)$ be the finite Dickson nearfield that arises from the pair $(3,2)$ and $v_1=(1,1,2,x+1,1), v_2=(0,0,0,2x+2,1),v_3=(1,1,1,x+2,1) \in R^5.$ By Theorem \ref{th2}, we have
\begin{align*}
gen(v_1,v_2,v_3)= \bigoplus_{i=1}^{4}u_iR,
\end{align*}
where $u_1=(1,1,0,0,0), u_2=(0,0,1,0,0), u_3=(0, 0,0,1,0)$ and $u_4=(0,0,0,0,1).$  
Note that $
gen(v_1,v_2,v_3)=  \big \{ \left( x ,x,y ,z,t \right)^T \vert \thickspace x,y ,z,t\in R  \big \}.$
However $gen(v_1,v_2,v_3)$ is not a subspace of $R^5.$ To see this, let $h \in gen(v_1,v_2,v_3), r \in R$ and $(r_1,r_2,r_3,r_4,r_5) \in R^5$ such that $r_1 \neq r_2.$
We have,
\begin{align*}
  \left( \left(r_1,r_2,r_3,r_4,r_5 \right) ^T +  \left( x ,x,y ,z,t \right)^T \right)r-\left(r_1,r_2,r_3,r_4,r_5 \right) ^T r  =  \big (  (r_1 +x)r-r_1r, \\ (r_2 +x)r-r_2r,  (r_3 +y)r-r_3r ,   (r_4 +z)r-r_4r, (r_5 +t)r-r_5r  \big )^T 
    \notin gen(v_1,v_2,v_3).
\end{align*}

\end{exa}

One fundamental property of \textit{gen}  is the following.
\begin{thm} Let $n \in \mathbb{N}$ and $R$ be a nearfield.
There exists vectors $v_1, \ldots, v_{n-1}$ of $R^n$ such that
\begin{align*}
gen(v_1, \ldots, v_{n-1})=R^n.
\end{align*} 
\label{ts}
\end{thm}
\begin{proof}
Let $\alpha_1,\alpha_2,\ldots,\alpha_{n-1} , \lambda \in R$  such that $(\alpha_1 + \alpha_2+ \cdots+\alpha_{n-1}  ) \lambda  \neq \alpha _1\lambda + \alpha_2 \lambda +\cdots +\alpha_{n-1} \lambda. $ We choose $v_1= (1,1,0, 0, \ldots,0)^T, v_2= (1,0,1, 0, \ldots,0)^T, \ldots, v_{n-1}= (1,0,0, \ldots,0,1)^T \in R^n.$  Then
  \begin{multline*}
  v  =  \left( (1,1,0, 0, \ldots,0)^T \alpha_1+  (1,0,1, 0, \ldots,0)^T \alpha_2  + \cdots + (1,0,0, \ldots,0,1)^T \alpha_{n-1}\right)    \lambda -  \\ (1,1,0, 0, \ldots,0)^T \alpha_1 \lambda -    (1,0,1, 0, \ldots,0)^T  \alpha_2 \lambda - \cdots- \\  (1,0,0, \ldots,0,1) ^T \alpha_{n-1}  \lambda \\ = 
 \left(
   (\alpha_1+\alpha_2+\ldots+\alpha_{n-1}) \lambda  - \alpha_1 \lambda-\alpha_2 \lambda- \cdots -\alpha_{n-1} \lambda, 0, 0, \ldots, 0 \right)^T.
 \end{multline*}
 So $v=   \left( \gamma, 0,
   0, \ldots , 0 \right)^T$ where $ \gamma = (\alpha_1+\alpha_2+\cdots+\alpha_{n-1}) \lambda - \alpha_1 \lambda-\alpha_2 \lambda- \cdots -\alpha_{n-1} \lambda$.  It follows that 
$
 v \gamma^{-1}= (1,0,0, 0, \ldots,0) \in gen (v_1,v_2,\ldots,v_{n-1}).$ 
  
  Let $\left( x_1, x_2, x_3, \ldots, x_n \right)^T \in R^n.$ Then 
  \begin{align*}
\left( x_1, x_2, x_3, \ldots, x_n \right)^T  =& v_1x_2+v_2x_3+ \cdots+ v_{n-1}x_n- v \gamma^{-1} (x_n+x_{n-1}+x_{n-2} \\ & + \cdots+ x_2-x_1) \\
  =  & v_1x_2+v_2x_3+ \cdots v_{n-1}x_n- \big( (v_1 \alpha_1 +v_2 \alpha_2+\cdots  v_{n-1} \alpha_{n-1} ) \lambda \\& -v_1 \alpha_1 \lambda-v_2 \alpha_2 \lambda - \cdots -
    v_{n-1} \alpha_{n-1} \lambda \big ) \gamma^{-1} (x_n+x_{n-1}+ \\ &  \quad \quad \quad \quad  \quad \quad x_{n-2}+ \cdots+ x_2-x_1).
  \end{align*}
  
  It follows that $ \left( x_1, x_2, x_3, \ldots, x_n \right)^T \in LC_2(v_1, \ldots, v_{n-1}),$ so $ \left( x_1, x_2, x_3, \ldots, x_n \right)^T \in gen(v_1, \ldots, v_{n-1})$. Thus $ gen(v_1, \ldots, v_{n-1}) =R^n.$
\end{proof}

\begin{rem}Surprisingly, unlike in traditional vector spaces, by Theorem \ref{ts} there exists two vectors $v_1$ and $v_2$ such that $gen (v_1,v_2)=R^3.$ By experimental computations, there  exists  no two vectors $v_1$ and $v_2$ such that $gen(v_1,v_2)$ is the whole space with high enough dimension, for instance  $R^{n^{n^2}}.$ 
\end{rem}

We end this section with a description of $gen(m)$ for $m \in M_R$.
\begin{lem}
Let $M_R$ be a Beidleman near-vector space. Let $m \in M.$ Then
$ gen(m)=mR.$
\label{lemm2}
\end{lem}
\begin{proof}
$LC_n(m)=mR$ for all positive integers $n$. Using Theorem \ref{th2} we have that $gen(m)=mR.$
\end{proof}

\section{Classification of the subspaces of $R^n$}

The extended Gaussian elimination algorithm can be used to determine the span of a given set of vectors, which is defined as the smallest submodule containing all the given vectors. 
From Example \ref{example}, we know that   some $R$-subgroups are not necessarily subspaces of Beidleman near-vector spaces $R^n,$ where $R$ is a nearfield. In Theorem \ref{th3} we describe the smallest subspace of $R^n$ containing a given set of vectors. This  allows us to classify all the subspaces of $R^n$ in Corollary \ref{th4}.
\begin{thm}
Let $v_1,\ldots,v_k$ be vectors of $R^n$. Then
\begin{align*}
span(v_1,\ldots,v_k)= \bigoplus_{i=1}^{k''}e_iR,
\end{align*}where $e_i$ is a row vector with only one non-zero entry \say{$1$} in its $i$-th position obtained from $v_i$ by an explicit procedure.
\label{th3}
\end{thm}
\begin{proof}

Let $v_1,\ldots,v_k \in R^n$. Since a submodule is an $R$-subgroup, we have
\begin{align*}
gen(v_1,\ldots,v_k )\subseteq span (v_1,\ldots,v_k).
\end{align*}  Note that  $gen(v_1,\ldots,v_k )$ will be a subspace if $gen(v_1,\ldots,v_k )=span(v_1,\ldots,v_k ).$

Let $u_1,\ldots,u_{k'}$ be determined as before in Theorem \ref{th2}.
\begin{itemize}
\item[Case 1:]Suppose  that  $u_i$ for all $i \in \{1 ,\ldots,k' \}$ has only one non-zero component.  Then 
$u_1R+u_2R+\ldots+u_{k'}R $ is a submodule of $R^n$, hence it is the desired span. So
\begin{align*}
span(v_1,\ldots,v_k ) &= gen(v_1,\ldots,v_k ) \\ 
 & =  \bigoplus_{i=1}^{k'}u_iR. 
\end{align*}
\item[Case 2:] Suppose that no row has more than two non-zero entries and  $u_i$ is the first $i$-th row that has entries $u_i^{j_1}\neq 0 \neq u_i^{j_2}$ i.e., $u_i= \big (0, \ldots,0, u_i^{j_1}, 0, \ldots, 0, u_i^{j_2}, 0 , \ldots, 0 \big )$. We  apply what we will call the \say{\textsl{adjustment trick}}. Let $ \alpha, \beta, \lambda \in R$ such that $(\alpha +\beta) \lambda \neq \alpha  \lambda  + \beta  \lambda .$
Define $m \in R^n$ by $m^j=\alpha \delta_{jj_2}$ for $1 \leq j \leq n$ where $\delta_{ij}$ is the Kronecker function. So
\begin{align*}
m^j =
\begin{cases}
0 \thickspace \text{if $j \neq j_2$}  \\
\alpha  \thickspace \text{if $j = j_2$}  
\end{cases}. 
\end{align*}
It follows that $m=(0, \ldots,0,\alpha, 0 , \ldots, 0).$ 

Define $a=u_i  ((u_i^{j_2})^{-1}\beta)= \big (0, \ldots,0,u_i^{j_1}(u_i^{j_2})^{-1}\beta, 0, \ldots, 0, \beta, 0, \ldots, 0  \big )$, so that $a^{j_2}=\beta$. 

Then let $v\in R^n$ be defined as
\begin{align*}
	v & = (m+a) \lambda - m  \lambda \\
	&=  \big (0, \ldots,0,u_i^{j_1}(u_i^{j_2})^{-1}\beta \lambda,  0, \ldots, 0,(\alpha +\beta) \lambda - \alpha \lambda, 0, \ldots, 0  \big ).
\end{align*}

By the additional condition on submodules, we must have $v\in span(u_1,\ldots,u_{k'})$. Hence we may add $v$ without changing the span (it strictly increases the gen though). Note that by construction, $v\neq a \lambda$. In fact, the only non-zero entry of $v-a  \lambda$ is the $j_2$ component. Hence we may reduce the $j_2$ entry of $u_i$ to zero. Since $a \in span (u_1,\ldots,u_{k'}),$ 
\begin{align*}
v-a \lambda & = \big (0, \ldots,0, (\alpha +\beta) \lambda - \alpha \lambda -\beta \lambda, 0, \ldots, 0  \big ) \\ & = \big (0,0, \ldots,0,\gamma , 0, \ldots, 0  \big ) \in span (u_1,\ldots,u_{k'}) \\ \thickspace \mbox{where} & \thickspace \gamma =(\alpha +\beta) \lambda - \alpha \lambda -\beta \lambda \thickspace \mbox{is in the $j_2$-th position of} \thickspace  v-a \lambda.
\end{align*} So $(v-a \lambda) \gamma ^{-1}= (0, \ldots,0,1,0 \ldots, 0)$ is denoted by $e_{j_2}$ (the row that only has a non-zero entry \say{$1$}  in its $j_2$-th column) and $e_{j_2} \in span (u_1,\ldots,u_{k'})$. Also $ \big ( u_i (u_i^{j_2})^{-1}- e_{j_2} \big ) \big ( u_i^{j_1}(u_i^{j_2})^{-1}\big )^{-1} = (0, \ldots,0,u_i^{j_1}(u_i^{j_2})^{-1},0,\ldots,0) \big (u_i^{j_1}(u_i^{j_2})^{-1} \big )^{-1}$ is denoted as $e_{j_1}$ (the row that  only has a non-zero entry \say{$1$} in its $j_1$-th column) and $e_{j_1}\in span (u_1,\ldots,u_{k'}).$ We now add the new rows $e_{j_1},e_{j_2}$ to the rows $u_1,\ldots,u_i, \ldots, u_{k'}$ and remove the row $u_i$ (the span is unchanged). We have that 
\begin{align*}
span(v_1,\ldots,v_k ) &= span (u_1,\ldots,u_i, \ldots, u_{k'})\\ & =  span (u_1,\ldots,e_{j_1},e_{j_2},u_i, \ldots, u_{k'}) \\
&= span (u_1,\ldots,e_{j_1},e_{j_2}, \ldots, u_{k'}). 
\end{align*}
Continuing the implementations of \say{\textsl{adjustment trick}} on the other rows $u_t$ (which has also two non-entries) where $t>i$, we may eliminate occurrences of multiple non-zero entries in the $u_t$ while appending new vectors with only one non-zero entry to make up for them. Thus
\begin{align*}span(v_1,\ldots,v_k ) &= span (u_1,\ldots,u_i, \ldots, u_{k'})\\
&= span(e_1,e_1,e_2, \ldots, e_{k''})\\
&=\bigoplus_{i=1}^{k''}e_iR,
\end{align*}
where $e_i$ is the $i$-th row that has a non-zero entry \say{$1$} in  its  $i$-th position only.
\item[Case 3:] Suppose that there exists at least  one row which has more than two non-zero entries and  $u_i$ is the first $i$-th row  with  $l$ non-zero entries where $l \geq 2$. Then continuing with the procedure in Case $2$, we must apply the  \say{\textsl{adjustment trick}} on the first two non-zero entries of $u_i$ and repeating the procedure on the other non-zero entries. Thus  we may eliminate occurrences of multiple non-zero entries in the $u_i$ while appending new vectors with only one non-zero entry to make up for them. Hence
\begin{align*}
span(v_1,\ldots,v_k ) &= span (u_1,\ldots,u_i, \ldots, u_{k'}) \\
& =  span (u_1,\ldots,u_i,e_{j_1},e_{j_2}, \ldots, e_{j_l}, \ldots, u_{k'})\\
&=span (u_1,\ldots,e_{j_1},e_{j_2}, \ldots, e_{j_l}, \ldots, u_{k'}) \\
&= span(e_1,e_1,e_2, \ldots, e_{k''})\\
&=\bigoplus_{i=1}^{k''}e_iR,
\end{align*} where $e_i$ is the $i$-th row that has a non-zero entry \say{$1$} in  column $i$ only. Thus we have proved that the span can always be written as a direct sum of vectors with only one non-zero entry.
\end{itemize}

\end{proof}
\begin{rem} Let $v_1,\ldots,v_k$ be finite number of vectors in $R^n$ arranged in a matrix $V$. Since $span(v_1,\ldots,v_k ) $ is the row space of the matrix $V,$ by Theorem \ref{th3} the dimension of $span(v_1,\ldots,v_k ) $ is $k''$.
\end{rem}
From the description of subspaces of $R^n$ in Theorem \ref{th3}, we deduce the following.

\begin{cor}
The subspaces of $R^n$ are all of the form $ S_1 \times S_2 \times \cdots \times S_n$ where $S_i= \{ 0 \}$ or $S_i=R$ for $i=1, \ldots,n$.
\label{th4}
\end{cor}
\begin{proof} 
Let $S$ be a subspace of $R^n.$ If $S= \{ 0 \}$ or $S=R^n$ then $S$ is the trivial subspace of $R^n.$ Let $v_1 \in S \setminus \{ 0 \},$ then $span (v_1) \subseteq S.$ Let $v_2 \in S \setminus span(v_1),$ then $span(v_1,v_2) \subseteq S.$ We continue this process until for $v_k \in S \setminus span(v_1,\ldots,v_{k-1} ) $ we have $span(v_1,\ldots,v_k )=S.$ Thus by Theorem \ref{th3}, $S$ is isomorphic (up to the reordering of coordinates) to $R^{k''}$ for $k'' \leq n.$ It is follows that $S$ is of the form
$ S_1 \times S_2 \times \cdots \times S_n$ where $S_i=R$ for $i=1,\ldots,k''$ and $S_i= \lbrace 0 \rbrace$ for $i=k''+1, \ldots,n.$  Hence each $S_i$ for $i=1,\ldots,n$ is subspace of $R$. But by Theorem \ref{irre} the only subspaces of $R$ is  $\lbrace 0 \rbrace$ and $R$. Hence $S_i= \lbrace 0 \rbrace $ or $S_i = R$, for $i =1, \ldots,n.$ 
\end{proof}

\begin{rem}
Let us consider $T= \lbrace (r,r, \ldots, r)  \in R^n | r \in R \rbrace$. $T$ is not of the form prescribed in  Corollary \ref{th4}. To see that $T$ is not subspace, let $ (r_1,0, \ldots, 0),(r,r, \ldots, r)   \in R^n$ and $ \alpha \notin  R_d$. Then 
\begin{align*} 
\left(   (r_1,0, \ldots, 0)^T  \ + (r,r, \ldots, r)^T  \ \right)   \alpha - \left( (r_1,0, \ldots, 0)^T \right) \alpha \\
  =  \big ( (r_1+r) \alpha +r_1 \alpha , r \alpha, \ldots ,  r \alpha  \big). 
\end{align*}
Since $ \alpha \notin R_d,$ there exists $ x, y \in R$ such that $(x+y) \alpha \neq x \alpha + y \alpha$. We choose $r_1=x$ and $r=y.$ So $(r_1+r) \alpha \neq r_1 \alpha  + r \alpha $ from which it follows that $ (r_1+r) \alpha -r_1 \alpha  \neq r \alpha$. Hence $T$ is not subspace of $R^n$.
\end{rem}

Let us illustrate Theorem  \ref{th3} with the following two examples. 
\begin{exa}
Let $R=DN(3,2)$ and  $v_1,v_2,v_3 \in R^5$ such that $v_1=(0,1,1,0,0), v_2=(0,x+1,2,0,x+1)$ and $v_3= (1,x+1,1,0,x).$ By Theorem \ref{th2},
\begin{align*}
gen(v_1,v_2,v_3) & = u_1R \oplus  u_2R \oplus  u_3R \oplus  u_4R,
\end{align*} where $u_1=(1,0,0,0,0), u_2=(0,1,0,0,0), u_3=(0,0,1,0,0)$ and $u_4=(0,0,0,0,1).$ By Theorem \ref{th3}, case 1, we have 
\begin{align*}
gen(v_1,v_2,v_3)& =span(v_1,v_2,v_3) \\
&= \bigoplus_{i=1}^4 e_iR  \cong R^4 \thickspace \mbox{where} \thickspace e_i=u_i , \thickspace \mbox{for} \thickspace i \in \{1, \ldots, 4 \},
\end{align*} 
is subspace of $R^5.$
\end{exa}
\begin{exa}
Let $R=DN(3,2)$ and  $v_1,v_2,v_3 \in R^5$ such that $v_1=(1,1,2,x+1,1), v_2=(0,0,0,2x+2,1)$ and $v_3= (1,1,1,x+2,1).$ By Theorem \ref{th2}
\begin{align*}
gen(v_1,v_2,v_3) & = u_1R \oplus  u_2R \oplus  u_3R \oplus  u_4R,
\end{align*} where $u_1=(1,1,0,0,0), u_2=(0,1,0,0,0), u_3=(0,0,1,0,0)$ and $u_4=(0,0,0,0,1).$ We wish to determine  $span(v_1,v_2,v_3).$ By Theorem \ref{th3}, case 2, we may apply the \say{\textsl{adjustment trick}} to reduce $u_1$. Define $m \in R^5$ such that 
\begin{align*}
m^j =
\begin{cases}
0 \thickspace \text{if $j \neq 1$}  \\
\alpha  \thickspace \text{if $j = 1.$}  
\end{cases}. 
\end{align*}

Let $ \alpha, \beta, \lambda \in R$ such that $(\alpha +\beta) \lambda \neq \alpha  \lambda  + \beta  \lambda .$
  Let $a= u_1 \big ( (u_1^1)^{-1} \big) \beta=u_1 \beta$  and $S=span (u_1,u_2,u_3,u_4).$  We have
\begin{align*}(m +a) \lambda -m \lambda & =
\left( (1,0,0,0,0) ^T  \alpha + (1,1,0,0,0) ^T \beta   \right) - \left( (1,0,0,0,0) ^T \right) \alpha   \\
& = \left(
   (\alpha + \beta) \lambda -\alpha \lambda ,
   \beta \lambda  , 0 ,0,0 \right)^ T \in S
\end{align*}
and
\begin{align*}
\left(
   (\alpha + \beta) \lambda -\alpha \lambda ,
   \beta \lambda  , 0 ,0,0 \right)^ T  - \left( \beta \lambda ,
   \beta \lambda  , 0 ,0,0 \right)^ T & = \left(
   (\alpha + \beta) \lambda -\alpha \lambda -\beta \lambda ,
   0 , 0 ,0,0 \right)^ T) \\
   & =  ( \gamma, 0,0,0,0) ^T \in S. 
\end{align*} where $ \gamma =  (\alpha + \beta) \lambda -\alpha \lambda -\beta \lambda$.
It follows that 
\begin{align*}
( \gamma, 0,0,0,0) ^T \gamma^{-1} = (1, 0,0,0,0) ^T = e_1 \in S. 
\end{align*}
Also
\begin{align*}
u_1-e_1=e_2 \in S.
\end{align*}
Therefore
\begin{align*}
span(v_1,v_2,v_3)& =span(e_1,e_2,u_1,u_2,u_3,u_4)  \\
& =span(e_1,e_2,u_2,u_3,u_4) \\
& = span(e_1,e_2,e_3,e_4,e_5) \\
&= \bigoplus_{i=1}^5 e_iR  \cong R^5,
\end{align*}
where $u_2=e_3, u_3=e_4$ and $u_4=e_5.$
\end{exa}
Surprisingly, unlike traditional vector spaces,  $span(v)$ for $v \in R^n$ can be the whole space.

\begin{pro} Let $v \in R^n$ and $k \leqslant n.$ Then 
$span(v)$ is $k$-dimensional if and only if $v$ contains $k$ non-zero entries.
\label{span}
\end{pro}
\begin{proof}Let $V \in R^{1 \times n} $ be a matrix of size $1 \times n$ which contains the entries of $v$.
Suppose $span(v)$ is $k$-dimensional. Then  the row space of the matrix $V$  has  standard basis $\{ e_1,\ldots, e_k \}.$ So $span(v)=\bigoplus_{i=1}^{k}e_iR.$ Then $v=\sum_{i=1}^{k}e_ir_i$ where $r_i  \in R.$ Thus $v$ contains $k$ non-zero entries. 
Conversely suppose $v$ contains $k$ non-zero entries.  We now implement the \say{\textsl{adjustment trick} } (see case $3$, proof of Theorem \ref{th3}). Hence we eliminates occurrences of multiples of non-zero entries in $v$ while appending $k$ new vectors with each containing only one non-zero entry. Thus $span(v)=\bigoplus_{i=1}^{k}e_iR.$ Therefore $span(v)$ is $k$-dimensional.
\end{proof}
\begin{cor}
Let $v \in R^n.$ Then $span(v)=vR$ if and only if $v$ has at most one non-zero entry. 
\end{cor}

\begin{proof}
If  $span(v)=vR$ then $span(v)$ is one dimensional and by Proposition \ref{span} $v$ must contain only one non-zero entry.   Suppose $v$ has at most one non-zero entry. By Lemma  \ref{lemm2} $gen(v)=vR.$ Thus $span(v)=gen(v)=vR.$
\end{proof}

In \cite{goldman1970foundations,knuth1971subspaces} the authors derived an expression that evaluates the number of $k$-dimensional subspaces of the finite dimensional vector space $(F^n,F)$ where $F$ is a finite field. In the case of the finite dimensional Beidleman near-vector space $(R^n,R)$, where $R$ is nearfield we have that
\begin{pro}
The number of subspaces of dimension $k$ of $R^n$ is $\Vector{n}{k}.$
\end{pro}
\begin{proof}
By Corollary \ref{th4} a $k$-dimensional subspace  is isomorphic to $R^k$ and $n-k$ is the number of zeros appearing in the $n$ coordinates. Now the number of subspaces of dimension $k$  corresponds to choosing  $k$ out of the $n$ coordinates, hence it is $\Vector{n}{k}.$
\end{proof}

\begin{cor}
There are  $2^n$  subspaces of $R^n$.
\end{cor}

 \section{Conclusion and open problem}

In future work, we suggest investigating the following question. Given a matrix $M$ in extended reduced row echelon form with non-zero rows, can we determine the  minimal sets of non-zero row vectors that form a matrix $N$ such that 
\textit{gen}( rows of M)=  \textit{gen}( rows of $N)$?   It might be also fruitful to investigate on the subspaces of $R^n$ in geometrical  view.

\textbf{Acknowledgments:}
Both authors are grateful  to Georg Anegg for his  helpful inputs  and  Dr Gareth Boxall for his advise. The first author worked on this paper while studying toward his PhD at Stellenbosch University. 

\textbf{Funding:} Both authors are grateful for funding by NRF, AIMS, Stellenbosch University and DAAD.
\section{Appendix}

\begin{algo} The expanded Gaussian  elimination (eGe) computes the smallest $R$-subgroup containing a given set of vectors in $R^n.$ The algorithm implements the normal Gaussian  elimination plus the distributivity trick.
\end{algo}
\begin{itemize}
\item[ \textbf{Input:} ] $v_1,v_2, \ldots,v_k \in R^n$ for $R =DN(q,m)$ where $m>1$ arranged in a matrix $V=(v_i^j)_{1 \leq i \leq k,1 \leq j \leq n} \in R^{k \times n}.$
\begin{itemize}

\item[ \textbf{Step 1}]$W= RREF(V)$ (reduced row echelon form of $V$ by Gaussian  elimination operations).

\item[ \textbf{Step 2}] ~ 

\begin{itemize}
\item[Case 1]Suppose that every column of $W$ has at most one non-zero entry. Then $gen(v_1,\ldots,v_k)= \bigoplus_{i=1}^{k}w_iR$.
\item[Case 2]Suppose that $j$-th column is  the first and the only column of  $W$ that has at least two non-zero entries denoted by $w_r^j,w_s^j,w_t^j$, \ldots where $r<s<t< \cdots$
\begin{itemize}
\item[Subcase 1] Consider the first two non-zero entries say $w_r^j\neq 0 \neq w_s^j$ with $r<s$.  \\ Apply the first \say{\textsl{distributivity trick}} by creating the new row $\phi.$
\item[Subcase 2] Consider the second two non-zero entries say $\phi^j \neq 0 \neq w_t^j.$\\
Apply the  second \say{\textsl{distributivity trick} }.
\end{itemize}
 \quad \vdots \\
  Continuing this manner until in the $j$-th column we have only one non-zero entry.
  \item[Case 3] Suppose  that there are more than one non-zero entries at each  $j$-th, $(j+1)$-th, $(j+2)$-th, \ldots, $n$-th columns of $W$.  Then apply the \say{\textsl{distributivity trick} } on every of the $j$-th, $(j+1)$-th, $(j+2)$-th, \ldots, $n$-th column until we have only one non-zero entry in every column.
\end{itemize}
\end{itemize}
\item[ \textbf{Output:} ] The final matrix  $U=(u_i^j)_{1 \leq i \leq k',1 \leq j \leq n} \in R^{k' \times n}$ which has at most one non-zero entry in every column. Let $u_1, \ldots, u_{k'}$ be the rows of $U$. We have, $gen(v_1,\ldots,v_k)= \bigoplus_{i=1}^{k'}u_iR$.

\end{itemize}

\begin{algo} The adjustment expanded Gaussian  elimination (aeGe) computes the smallest subspace containing a  given set of vectors in $R^n.$ The algorithm implement
the expanded Gaussian  elimination plus the adjustment trick.

\end{algo}
\begin{itemize}
\item[ \textbf{Input:} ] $v_1,v_2, \ldots,v_k \in R^n$ for $R =DN(q,m)$ where $m>1$ arranged in a matrix $V=(v_i^j)_{1 \leq i \leq k,1 \leq j \leq n} \in R^{k \times n}.$
\begin{itemize}

\item[ \textbf{Step 1}]  Apply eGe and we get $gen(v_1,\ldots,v_k)=\bigoplus_{i=1}^{k'}u_iR.$
\item[ \textbf{Step 2}] ~ 

\begin{itemize}
\item[Case 1] Suppose that $u_i$ for $i=1,\ldots,k'$ have exactly one non-zero entry. Then $span(v_1,\ldots,v_k)=\bigoplus_{i=1}^{k'}u_iR.$
\item[Case 2]Suppose that $u_i$ is the only $i$-th row that has more than one non-zero entries denoted $u_i^{j_1},u_i^{j_2}, u_i^{j_3}, \ldots$. Apply the \say{\textsl{adjustment trick} } on the row $u_i$ until we eliminate occurrences of multiple non-zero
entries while appending new vectors with only one non-zero entry.
\item[Case 3] Suppose that the rows $u_i, u_{i+1}, \ldots,u_{k'}$ have each more than one non-zero entry. Apply on each row the \say{\textsl{adjustment trick} }.
\end{itemize}
\end{itemize}
\item[ \textbf{Output:} ] The final matrix  $E=(e_i^j)_{1 \leq i \leq k'',1 \leq j \leq n} \in R^{k'' \times n}$ which has at most one non-zero entry in every column. Let $e_1, \ldots, e_{k'}$ be the rows of $E$. We have, $gen(v_1,\ldots,v_k)= \bigoplus_{i=1}^{k''}e_iR$.

\end{itemize}

%
%
%
%
%
%
%
%
%

%

\end{document}